\pdfoutput=1
\documentclass[a4paper,12pt,reqno,oneside]{amsart}
\usepackage[utf8x]{inputenc}
\usepackage[T1]{fontenc}
\usepackage{geometry}
\usepackage{lmodern}

\usepackage{amsmath}
\usepackage{amssymb}
\usepackage{amsthm}
\usepackage{array}
\usepackage{enumerate}
\PassOptionsToPackage{hyphens}{url}
\usepackage[bookmarksnumbered,colorlinks,linkcolor=black,citecolor=black,urlcolor=black]{hyperref}
\usepackage[capitalise]{cleveref}

\newcommand{\bC}{\mathbb{C}}
\newcommand{\bQ}{\mathbb{Q}}
\newcommand{\bR}{\mathbb{R}}
\newcommand{\bZ}{\mathbb{Z}}

\newcommand{\cA}{\mathcal{A}}
\newcommand{\cF}{\mathcal{F}}
\newcommand{\cH}{\mathcal{H}}
\newcommand{\Ag}{\mathcal{A}_g}
\newcommand{\Fg}{\mathcal{F}_g}
\newcommand{\Hg}{\mathcal{H}_g}

\newcommand{\fS}{\mathfrak{S}}

\newcommand{\gB}{\mathbf{B}}
\newcommand{\gG}{\mathbf{G}}
\newcommand{\gGL}{\mathbf{GL}}
\newcommand{\gGSp}{\mathbf{GSp}}
\newcommand{\gH}{\mathbf{H}}
\newcommand{\gL}{\mathbf{L}}
\newcommand{\gM}{\mathbf{M}}
\newcommand{\gN}{\mathbf{N}}
\newcommand{\gO}{\mathbf{O}}
\newcommand{\gP}{\mathbf{P}}
\newcommand{\gQ}{\mathbf{Q}}
\newcommand{\gS}{\mathbf{S}}
\newcommand{\gSL}{\mathbf{SL}}
\newcommand{\gSO}{\mathbf{SO}}
\newcommand{\gSp}{\mathbf{Sp}}
\newcommand{\gT}{\mathbf{T}}
\newcommand{\gU}{\mathbf{U}}
\newcommand{\gZ}{\mathbf{Z}}

\newcommand{\rH}{\mathrm{H}}
\newcommand{\rM}{\mathrm{M}}

\newcommand{\AG}[1]{A_{\gG,#1}}
\newcommand{\AH}[1]{A_{\gH,#1}}
\newcommand{\KG}{K_\gG}
\newcommand{\KH}{K_\gH}
\newcommand{\KZ}{K_\gZ}
\newcommand{\gBZ}{\gB_\gZ}
\newcommand{\gMG}{\gM_\gG}
\newcommand{\gMH}{\gM_\gH}
\newcommand{\gPG}{\gP_\gG}
\newcommand{\gPH}{\gP_\gH}
\newcommand{\gPZ}{\gP_\gZ}
\newcommand{\gSG}{\gS_{\gG}}
\newcommand{\gSH}{\gS_{\gH}}
\newcommand{\gSZ}{\gS_\gZ}
\newcommand{\gTG}{\gT_\gG}
\newcommand{\gUG}{\gU_\gG}
\newcommand{\gUH}{\gU_\gH}
\newcommand{\gUZ}{\gU_\gZ}

\newcommand{\abs}[1]{\lvert #1 \rvert}

\newcommand{\defterm}[1]{\textbf{#1}}

\newtheorem{lemma}{Lemma}[section]
\newtheorem{proposition}[lemma]{Proposition}
\newtheorem{theorem}[lemma]{Theorem}
\newtheorem{conjecture}[lemma]{Conjecture}
\newtheorem*{lemma4A}{Lemma 4.A}

\usepackage{ifthen}
\newcounter{constant}
\newcommand{\newC}[1]{%
   \refstepcounter{constant} C_{\theconstant}%
   \ifthenelse{\equal{#1}{*}} { } {%
      \label{C:#1}%
   }%
}
\newcommand{\refC}[1]{C_{\ref*{C:#1}}}

\title{Height bounds and the Siegel property}
\author{Martin Orr}

\begin{document}

\begin{abstract}
Let \( \gG \) be a reductive group defined over \( \bQ \) and let \( \fS \) be a Siegel set in \( \gG(\bR) \).
The Siegel property tells us that there are only finitely many \( \gamma \in \gG(\bQ) \) of bounded determinant and denominator for which the translate \( \gamma.\fS \) intersects~\( \fS \).
We prove a bound for the height of these \( \gamma \) which is polynomial with respect to the determinant and denominator.
The bound generalises a result of Habegger and Pila dealing with the case of \( \gGL_2 \), and has applications to the Zilber--Pink conjecture on unlikely intersections in Shimura varieties.

In addition we prove that if \( \gH \) is a subgroup of \( \gG \), then every Siegel set for \( \gH \) is contained in a finite union of \( \gG(\bQ) \)-translates of a Siegel set for \( \gG \).
\end{abstract}

\maketitle

\section{Introduction}

A Siegel set is a subset of the real points \( \gG(\bR) \) of a reductive \( \bQ \)-algebraic group of a certain nice form.
The notion of Siegel set was introduced by Borel and Harish-Chandra~\cite{borel-harish-chandra}, in order to prove the finiteness of the covolume of arithmetic subgroups of \( \gG(\bR) \).
In this paper we use a variant of the notion due to Borel~\cite{borel:groupes-arithmetiques} which takes into account the \( \bQ \)-structure of the group~\( \gG \), and gives an intrinsic construction of fundamental sets for arithmetic  subgroups in \( \gG(\bR) \).

Let \( \fS \subset \gG(\bR) \) be a Siegel set (see section~\ref{sec:siegel-sets} for the precise definition).
The primary theorem of this paper is a bound for the height of elements of
\[ \fS.\fS^{-1} \cap \gG(\bQ) = \{ \gamma \in \gG(\bQ) : \gamma.\fS \cap \fS \neq \emptyset \} \]
in terms of their determinant and denominators.
This gives a quantitative version of \cite[Corollaire~15.3]{borel:groupes-arithmetiques},
which asserts that \( \fS.\fS^{-1} \cap \gG(\bQ) \) has only finitely many elements with given determinant and denominators.
This in turn implies a quantitative version of the Siegel property, one of the key properties of Siegel sets.

\begin{theorem} \label{intro:main-theorem}
Let \( \gG \) be a reductive \( \bQ \)-algebraic group and let \( \fS \subset \gG(\bR) \) be a Siegel set.
Let \( \rho \colon \gG \to \gGL_n \) be a faithful \( \bQ \)-algebraic group representation.

There exists a constant \( \newC{height-bound-intro-multiplier} \) (depending on \( \gG \), \( \fS \) and \( \rho \)) such that,
for all
\[ \gamma \in \fS.\fS^{-1} \cap \gG(\bQ) , \]
if \( N = \abs{\det \rho(\gamma)} \) and \( D \) is the maximum of the denominators of entries of \( \rho(\gamma) \), then
\[ \rH(\rho(\gamma)) \leq \max(\refC{height-bound-intro-multiplier} N D^n, D). \]
\end{theorem}

This theorem was inspired by a result of Habegger and Pila \cite[Lemma~5.2]{habegger-pila:beyond-ao}.
They dealt with the case \( \gG = \gGL_2 \), as a step in proving some cases of the Zilber--Pink conjecture on unlikely intersections in \( Y(1)^n \).
We are motivated by applications of \cref{intro:main-theorem} to the Zilber--Pink conjecture in higher-dimensional Shimura varieties, which is the subject of work in progress by the author.
The key point for these applications is that the bound is polynomial in the determinant~\( N \).

The second main theorem of this paper compares Siegel sets for the group~\( \gG \) with Siegel sets for a subgroup \( \gH \subset \gG \), which can be seen as a result on the functoriality of Siegel sets with respect to injections of \( \bQ \)-algebraic groups.
This theorem is used in the proof of \cref{intro:main-theorem} to reduce to the case \( \gG = \gGL_n \).
It also has its own applications to the Zilber--Pink conjecture.

\begin{theorem} \label{intro:siegel-set-inclusion}
Let \( \gG \) and \( \gH \) be reductive \( \bQ \)-algebraic groups, with \( \gH \subset \gG \).
Let \( \fS_\gH \) be a Siegel set in \( \gH(\bR) \).

Then there exist a finite set \( C \subset \gG(\bQ) \) and a Siegel set \( \fS_\gG \subset \gG(\bR) \) such that
\[ \fS_\gH \subset C.\fS_\gG. \]
\end{theorem}

\Cref{siegel-set-inclusion} gives some additional information about how the Siegel sets \( \fS_\gG \) and \( \fS_\gH \) are related to each other (in terms of the associated Siegel triples).

\subsection{Previous results:\ height bounds}

The primary inspiration for \cref{intro:main-theorem} is the following result of Habegger and Pila.

\begin{proposition} \cite[Lemma~5.2]{habegger-pila:beyond-ao} \label{hp-lemma}
Let \( \cF \) denote the standard fundamental domain for the action of \( \gSL_2(\bZ) \) on the upper half-plane.

There exists a constant \( \newC{hp-multiplier} \) such that:
for all points \( x, y \in \cF \), if the associated elliptic curves are related by an isogeny of degree~\( N \), then there exists \( \gamma \in \rM_2(\bZ) \) such that
\[ \gamma x = y, \det \gamma = N \text{ and } \rH(\gamma) \leq \refC{hp-multiplier} N^{10}. \]
\end{proposition}

In order to relate \cref{hp-lemma} to \cref{intro:main-theorem}, recall that the upper half-plane \( \cH \) can be identified with the symmetric space \( \gGL_2(\bR)^+ / \bR^\times \gSO_2(\bR) \), with \( \gGL_2(\bR)^+ \) acting on \( \cH \) by Möbius transformations.
Under this identification, the standard fundamental domain
\[ \cF = \{ z \in \cH : -\tfrac{1}{2} \leq \operatorname{Re} z \leq \tfrac{1}{2}, \; \abs{z} \geq 1 \} \]
is contained in the image of the standard Siegel set
\[ \fS = \Omega_{1/2} A_{\sqrt{3}/2} K \subset \gGL_2(\bR) \]
as defined in section~\ref{ssec:gln:siegel-sets}.

We further identify the quotient \( \gSL_2(\bZ) \backslash \cH \) with the moduli space \( Y(1) \) of elliptic curves over \( \bC \).
It is easy to prove that the elliptic curves associated with points \( x, y \in \cH \) are related by an isogeny of degree~\( N \) if and only if there exists \( \gamma \in \rM_2(\bZ) \) such that
\begin{equation} \label{eqn:isogeny-matrix}
\gamma x = y  \; \text{ and } \;  \det \gamma = N.
\end{equation}
\Cref{intro:main-theorem} tells us that any \( \gamma \) satisfying \eqref{eqn:isogeny-matrix} has height at most \( \refC{height-bound-intro-multiplier} N \), improving on the exponent~\( 10 \) which appears in \cref{hp-lemma}.

\Cref{intro:main-theorem} also implies a uniform version of the following previous result of the author (which is a combination of \cite[Lemma~3.3]{orr:andre-pink} with \cite[Theorem~1.3]{orr:polarised-isogenies}).

\begin{proposition} \label{andre-pink-lemma}
Let \( \cF_g \) denote the standard fundamental domain for the action of \( \gSp_{2g}(\bZ) \) on the Siegel upper half-space of rank \( g \).
Fix a point \( x \in \cF_g \).

There exist constants \( \newC{andre-pink-multiplier} \) and \( \newC{andre-pink-exponent} \) such that:
for all points \( y \in \cF_g \), if the principally polarised abelian varieties associated with \( x \) and \( y \) are related by a polarised isogeny of degree~\( N \), then there exists a matrix \( \gamma \in \gGSp_{2g}(\bQ)^+ \) such that
\[ \gamma x = y \text{ and } \rH(\gamma) \leq \refC{andre-pink-multiplier} N^{\refC{andre-pink-exponent}}. \]
\end{proposition}

In \cref{andre-pink-lemma}, the constant \( \refC{andre-pink-multiplier} \) depends on the fixed point \( x \in \Fg \) and only the other point~\( y \) is allowed to vary.
On the other hand, we can apply \cref{intro:main-theorem} to the symmetric space \( \Hg \) in a similar way to that sketched above for \( \cH \).
This gives a much stronger result in which the constant is uniform in both \( x \) and \( y \).
Hence \cref{intro:main-theorem} can be used to prove results on unlikely intersections in \( \Ag \times \Ag \) for which \cref{andre-pink-lemma} is not sufficient.

Note that \cite[Lemma~3.3]{orr:andre-pink} gives a height bound for unpolarised as well as polarised isogenies.
It is not possible to directly deduce a uniform version of this bound for unpolarised isogenies from \cref{intro:main-theorem} because \cite[Lemma~3.3]{orr:andre-pink} concerns the homogeneous space \( \gGL_{2g}(\bR)/\gGL_g(\bC) \) while \cref{intro:main-theorem} applies to the symmetric space \( \gGL_{2g}(\bR)/\bR^\times \gO_{2g}(\bR) \).

\subsection{Previous results:\ Siegel sets and subgroups}

Let \( \gH \) be a reductive \( \bQ \)-algebraic subgroup of \( \gG = \gGL_n \).
Borel and Harish-Chandra gave a recipe in \cite[Theorem~6.5]{borel-harish-chandra} for constructing a fundamental set for \( \gH(\bR) \) which is contained in a finite union of \( \gG(\bQ) \)-translates of a Siegel set for \( \gG \).
However it is not obvious how the resulting fundamental set is related to a Siegel set for \( \gH \).
\Cref{intro:siegel-set-inclusion} resolves this by directly relating Siegel sets for \( \gG \) and \( \gH \).

\Cref{intro:siegel-set-inclusion} can also be interpreted as a result about functoriality of Siegel sets.
According to a remark on \cite[p.~86]{borel:groupes-arithmetiques}, if \( f \colon \gH \to \gG \) is a \emph{surjective} morphism of reductive \( \bQ \)-algebraic groups and \( \fS_\gH \) is a Siegel set in \( \gH(\bR) \), then \( f(\fS_\gH) \) is contained in a Siegel set in \( \gG(\bR) \).
\Cref{intro:siegel-set-inclusion} gives a similar result for injective morphisms of reductive \( \bQ \)-algebraic groups, where the conclusion must be weakened to saying that the image of a Siegel set is contained in a finite union of \( \gG(\bQ) \)-translates of a Siegel set.
We can of course combine these to conclude that for an arbitrary morphism \( f \colon \gH \to \gG \), the image of a Siegel set \( \fS_\gH \subset \gH(\bR) \) is contained in a finite union of \( \gG(\bQ) \)-translates of a Siegel set in \( \gG(\bR) \).

The proof of \cref{intro:siegel-set-inclusion} gives an explicit bound for the size of the set \( C \subset \gG(\bQ) \), namely \( \# C \) is at most the size of the \( \bQ \)-Weyl group of \( \gG \).
The uniform nature of this bound is less powerful than it might at first appear because the Siegel set \( \fS_\gG \) depends on \( \fS_\gH \).

\subsection{Application to unlikely intersections}

The author's motivation for studying \cref{intro:main-theorem} is due to its applications to the Zilber--Pink conjecture on unlikely intersections in Shimura varieties \cite[Conjecture~1.2]{pink:zp}.
To illustrate these applications, consider the following special case of the Zilber--Pink conjecture.

\begin{conjecture} \label{zp-conj-isog}
Let \( g \geq 2 \) and let \( \Ag \) denote the moduli space of principally polarised abelian varieties of dimension~\( g \) over~\( \bC \).

For each point \( s \in \Ag \), let \( (A_s, \lambda_s) \) denote the associated principally polarised abelian variety.
Let
\[ \Sigma = \{ (s_1, s_2) \in \Ag \times \Ag : \text{there exists an isogeny } A_{s_1} \to A_{s_2} \}. \]

Let \( V \subset \Ag \times \Ag \) be an irreducible algebraic curve.

If \( V \cap \Sigma \) is infinite, then \( V \) is contained in a proper special subvariety of \( \Ag \times \Ag \).
\end{conjecture}

In \cite{habegger-pila:beyond-ao}, Habegger and Pila used \cref{hp-lemma} to prove a result similar to \cref{zp-conj-isog} but for the Shimura variety \( \cA_1^n \) (\( n \geq 3 \)) instead of \( \Ag \times \Ag \) (\( g \geq 2 \)) (for reasons of dimension, \cref{zp-conj-isog} is false for \( \cA_1 \times \cA_1 \)).

In work currently in progress, the author of this paper proves \cref{zp-conj-isog} subject to certain technical conditions and a restricted definition of the set~\( \Sigma \).
This work requires the uniform version of \cref{andre-pink-lemma} which is implied by the \( \gGSp_{2g} \) case of \cref{intro:main-theorem}.
Because \cref{intro:main-theorem} applies to all reductive groups, not just~\( \gGSp_{2g} \), it should also be useful for proving statements similar to \cref{zp-conj-isog} where \( \Ag \) is replaced by an arbitrary Shimura variety.
However, at present it is not known how to prove the Galois bounds which would be required for such a statement.

\subsection{Outline of paper}

Section~\ref{sec:siegel-sets} contains the definition of Siegel sets and the associated notation used throughout the paper.
In section~\ref{sec:gln:proof} we prove \cref{intro:main-theorem} for standard Siegel sets in \( \gGL_n \), and combine this with \cref{intro:siegel-set-inclusion} to deduce the general statement of \cref{intro:main-theorem}.
The proof of the \( \gGL_n \) case is entirely self-contained.
Finally section~\ref{sec:subgroups} contains the proof of \cref{intro:siegel-set-inclusion}, relying on results on parabolic subgroups and roots from \cite{borel-tits:groupes-reductifs}.

\subsection{Notation} \label{ssec:notation}

If \( \gG \) is a real algebraic group, then we write \( \gG(\bR)^+ \) for the identity component of \( \gG(\bR) \) in the Euclidean topology.

We use a naive definition for the height of a matrix with rational entries, as in \cite{pila-wilkie}: if \( \gamma \in \rM_n(\bQ) \), then its \defterm{height} is
\[ \rH(\gamma) = \max_{1 \leq i, j \leq n} \rH(\gamma_{ij}) \]
where the height of a rational number \( a/b \) (written in lowest terms) is \( \max(\abs{a}, \abs{b}) \).
For an algebraic group \( \gG \) other than \( \gGL_n \), we define the heights of elements of \( \gG(\bQ) \) via a choice of faithful representation \( \gG \to \gGL_n \).

In order to avoid writing uncalculated constant factors in every inequality in the proof of \cref{intro:main-theorem}, we use the notation
\[ X \ll Y \]
to mean that there exists a constant \( C \), depending only on the group \( \gG \), the representation \( \rho \) and the Siegel set \( \fS \), such that
\[ \abs{X} \leq C\abs{Y}. \]

\subsection*{Acknowledgements}

I thank Philipp Habegger and Jonathan Pila for their suggestion that I should study generalisations of their result \cite[Lemma~5.2]{habegger-pila:beyond-ao} (\cref{hp-lemma} in this paper).
This suggestion was the initial inspiration for this paper.
I am grateful to Christopher Daw, Gisele Teixeira Paula, Jonathan Pila, Jinbo Ren and Andrei Yafaev for useful discussions during the writing of the paper.
I am also grateful to the referee for suggestions which improved the paper.

The work which led to this paper was funded by European Research Council grant 307364 and by EPSRC grant EP/M020266/1.

\medskip

This paper was published in \textit{Algebra \& Number Theory}, 2018, vol.~12, no.~2, pp.~455--478 (DOI: \texttt{10.2140/ant.2018.12.455}), published by Mathematical Sciences Publishers.
This version of the paper contains corrections to some typos and minor errors in the published version.
I am grateful to Dave Witte Morris for bringing these errors to my attention.

I am grateful to Christian Schnell, who discovered an error in the proof of \cref{siegel-set-inclusion} in the published version of the paper, and suggested how to correct it.  This error is corrected in this version of the paper.  The error, some associated examples, and a strengthened version of \cref{siegel-set-inclusion} which is often useful for applications, are described in the published correction \cite{orr-schnell} by myself and Schnell.

\section{Definition of Siegel sets} \label{sec:siegel-sets}

The definitions of Siegel sets used by different authors (for example, \cite{borel:groupes-arithmetiques} and \cite{amrt:smooth-compactifications}) vary in minor ways, so we state here the precise definition used in this paper.
At the same time, we define the notation which we shall use in sections \ref{sec:gln:proof} and~\ref{sec:subgroups} for the various ingredients in the construction of Siegel sets.

\subsection{Standard Siegel sets in \texorpdfstring{\( \gGL_n \)}{GLn}} \label{ssec:gln:siegel-sets}

Before defining Siegel sets in general, we begin with the simpler special case of ``standard Siegel sets'' in \( \gGL_n \).
Our definition of standard Siegel sets follows \cite[D\'efinition~1.2]{borel:groupes-arithmetiques}.
Compared to \cite{borel:groupes-arithmetiques}, we use the reverse order of multiplication for elements of~\( \gGL_n \) and therefore reverse the inequalities in the definition of \( A_t \).

Make the following definitions (all of these are special cases of the corresponding notations for general Siegel sets):
\begin{enumerate}
\item \( \gP \subset \gGL_n \) is the Borel subgroup consisting of upper triangular matrices.
\item \( K = \gO_n(\bR) \) is the maximal compact subgroup consisting of orthogonal matrices.
\item \( \gS \subset \gP \) is the maximal \( \bQ \)-split torus consisting of diagonal matrices.
\item \( A_t \) is the set \( \{ \alpha \in \gS(\bR)^+ : \alpha_j/\alpha_{j+1} \geq t \text{ for all } j \} \) for any real number \( t > 0 \).
\item \( \Omega_u \) is the compact set 
\[ \{ \nu \in \gP(\bR) : \nu_{ii} = 1 \text{ for all } i \text{ and } \abs{\nu_{ij}} \leq u  \text{ for }  1 \leq i < j \leq n \} \]
for any real number \( u > 0 \).
\end{enumerate}

A \defterm{standard Siegel set} in \( \gGL_n \) is a set of the form
\[ \fS  =  \Omega_u A_t K  \subset  \gGL_n(\bR) \]
for some positive real numbers \( u \) and \( t \).

According to \cite[Th\'eor\`emes 1.4, 4.6]{borel:groupes-arithmetiques}, if \( t \leq \sqrt{3}/2 \) and \( u \geq \tfrac{1}{2} \), then \( \fS \) is a fundamental set for \( \gGL_n(\bZ) \) in \( \gGL_n(\bR) \).

\subsection{Definition of Siegel sets in general} \label{ssec:siegel-sets-definition}

Let \( \gG \) be a reductive \( \bQ \)-algebraic group.
In order to define a Siegel set in \( \gG(\bR) \), we begin by making choices of the following subgroups of~\( \gG \):

\begin{enumerate}
\item \( \gP \) a minimal parabolic \( \bQ \)-subgroup of \( \gG \);
\item \( K \) a maximal compact subgroup of \( \gG(\bR) \).
\end{enumerate}

\begin{lemma} \label{siegel-triple-unique-torus}
For any \( \gP \) and \( K \), there exists a unique \( \bR \)-torus \( \gS \subset \gP \) satisfying the conditions
\begin{enumerate}[(i)]
\item \( \gS \) is \( \gP(\bR) \)-conjugate to a maximal \( \bQ \)-split torus in \( \gP \).
\item \( \gS \) is stabilised by the Cartan involution associated with \( K \).
\end{enumerate}
\end{lemma}

\begin{proof}
This follows from the lemma in \cite[chapter~II, section~3.7]{amrt:smooth-compactifications}.
\end{proof}

We define a \defterm{Siegel triple} for~\( \gG \) to be a triple \( (\gP, \gS, K) \) satisfying the conditions of \cref{siegel-triple-unique-torus}.
We remark that these conditions could equivalently be stated as:
\begin{enumerate}[(i)]
\item \( \gS \) is a lift of the unique maximal \( \bQ \)-split torus in \( \gP/R_u(\gP) \).
\item \( \operatorname{Lie} \gS(\bR) \) is orthogonal to \( \operatorname{Lie} K \) with respect to the Killing form of~\( \gG \).
\end{enumerate}

Define the following further pieces of notation:
\begin{enumerate}
\item \( \gU \) is the unipotent radical of \( \gP \).
\item \( \gM \) is the preimage in \( Z_\gG(\gS) \) of the maximal \( \bQ \)-anisotropic subgroup of~\( \gP/\gU \).
(Note that by \cite[Corollaire~4.16]{borel-tits:groupes-reductifs}, \( Z_\gG(\gS) \) is a Levi subgroup of \( \gP \) and hence maps isomorphically onto \( \gP/\gU \).)
\item \( \Delta \) is the set of simple roots of \( \gG \) with respect to \( \gS \), using the ordering induced by \( \gP \).
(The roots of \( \gG \) with respect to \( \gS \) form a root system because \( \gS \) is conjugate to a maximal \( \bQ \)-split torus in~\( \gG \).)
\item \( A_t = \{ \alpha \in \gS(\bR)^+ : \chi(\alpha) \geq t \text{ for all } \chi \in \Delta \} \) for any real number \( t > 0 \).
\end{enumerate}

A \defterm{Siegel set} in \( \gG(\bR) \) (with respect to \( (\gP, \gS, K) \)) is a set of the form
\[ \fS = \Omega A_t K \]
where
\begin{enumerate}
\item \( \Omega \) is a compact subset of \( \gU(\bR) \gM(\bR)^+ \); and
\item \( t \) is a positive real number.
\end{enumerate}

\subsection{Comparison with other definitions} \label{ssec:siegel-sets-comparison}

In order to reduce confusion caused by definitions of Siegel sets which vary from one author to another, we explain how our definition compares with the definitions used in \cite{borel-harish-chandra}, \cite{borel:groupes-arithmetiques} and \cite{amrt:smooth-compactifications}.

First we compare with \cite[chapter~II, section~4.1]{amrt:smooth-compactifications}.

\begin{enumerate}
\item In \cite{amrt:smooth-compactifications}, Siegel sets are subsets of the symmetric space \( \gG(\bR)/K \), while for us they are \( K \)-right-invariant subsets of \( \gG(\bR) \).
These two perspectives are related by the quotient map \( \gG(\bR) \to \gG(\bR)/K \).

\item In \cite{amrt:smooth-compactifications}, \( \Omega \) is any compact subset of \( \gP(\bR) \), while we require \( \Omega \) to be contained in \( \gU(\bR) \gM(\bR)^+ \).
Every Siegel set in the sense of \cite{amrt:smooth-compactifications} is contained in a Siegel set in our sense and vice versa, so this difference does not matter in applications.
We impose the stricter condition on \( \Omega \) because it ensures that Siegel sets are related to the horospherical decomposition in \( \gG(\bR)/K \) (as explained in \cite[section~I.1.9]{borel-ji:symmetric-spaces}).
\end{enumerate}

Now we compare with \cite[D\'efinition~12.3]{borel:groupes-arithmetiques}.
Note that differences (3) and~(4) are significant.

\begin{enumerate}
\item We multiply together \( \Omega \), \( A_t \) and \( K \) in the opposite order from \cite{borel:groupes-arithmetiques}.
This change forces us to reverse the inequalities in the definition of \( A_t \).

\item In \cite{borel:groupes-arithmetiques}, \( \Omega \) is required to be a compact neighbourhood of the identity in \( \gU(\bR) \gM(\bR)^+ \) while we allow any compact subset.

\item Instead of our condition~(i) for~\( \gS \), \cite{borel:groupes-arithmetiques} imposes the condition that \( \gS \) must be a maximal \( \bQ \)-split torus in \( \gP \).
This stronger condition is inconvenient when we also impose condition~(ii), because there does not exist a maximal \( \bQ \)-split torus satisfying condition~(ii) for every choice of \( \gP \) and \( K \).
In particular, \cref{intro:siegel-set-inclusion} does not hold if \( \gS_\gG \) is required to be \( \bQ \)-split.

\item Our condition~(ii) for~\( \gS \) is not part of the definition of Siegel set in \cite{borel:groupes-arithmetiques}.
In \cite{borel:groupes-arithmetiques}, a Siegel set is called \defterm{normal} if condition~(ii) is satisfied.
We include condition~(ii) in the definition of a Siegel set because without it the Siegel property does not necessarily hold.
Indeed most of the theorems in \cite[chapter~15]{borel:groupes-arithmetiques} apply only to Siegel sets satisfying condition~(ii), even though the word ``normal'' is omitted from their statements.
Similarly this paper's \cref{intro:main-theorem} does not hold without condition~(ii) on~\( \gS \).
\end{enumerate}

The definition of ``Siegel domain'' in \cite[section~4]{borel-harish-chandra} is less fine than the definition used in this paper, or the one in \cite{borel:groupes-arithmetiques}, because it takes into account only the structure of \( \gG \) as a real algebraic group and not its structure as a \( \bQ \)-algebraic group.
Consequently \cite{borel-harish-chandra} could not use their Siegel domains directly to construct fundamental sets for arithmetic subgroups in \( \gG(\bR) \); instead they constructed such fundamental sets using an embedding of \( \gG \) into \( \gGL_n \) and standard Siegel sets in \( \gGL_n(\bR) \).

\subsection{Siegel sets and fundamental sets} \label{ssec:siegel-sets-fundamental}

The importance of Siegel sets is due to their use in constructing fundamental sets for an arithmetic subgroup \( \Gamma \) in \( \gG(\bR) \).
We say that a set \( \Omega \subset \gG(\bR) \) is a \defterm{fundamental set} for \( \Gamma \) if the following conditions are satisfied:
\begin{enumerate}[(F1)]
\setcounter{enumi}{-1}
\item \( \Omega.K = \Omega \) for a suitable maximal compact subgroup \( K \subset \gG(\bR) \);
\item \( \Gamma.\Omega = \gG(\bR) \); and
\item for every \( \theta \in \gG(\bQ) \),%
\footnote{Corrected from the published version.}
the set
\[ \{ \gamma \in \Gamma : \gamma.\Omega \cap \theta.\Omega \neq \emptyset \} \]
is finite (the Siegel property).
\end{enumerate}

The following two theorems show that, if we make suitable choices of Siegel set \( \fS \subset \gG(\bR) \) and finite set \( C \subset \gG(\bQ) \), then \( C.\fS \) is a fundamental set for \( \Gamma \) in \( \gG(\bR) \).

\begin{theorem} \cite[Th\'eor\`eme 13.1]{borel:groupes-arithmetiques} \label{borel-siegel-surj}
Let \( \Gamma \) be an arithmetic subgroup of \( \gG(\bQ) \).
Let \( (\gP, \gS, K) \) be a Siegel triple for \( \gG(\bR) \).

There exist a Siegel set~\( \fS \subset \gG(\bR) \) with respect to \( (\gP, \gS, K) \) and a finite set \( C \subset \gG(\bQ) \) such that
\[ \gG(\bR) = \Gamma.C.\fS. \]
\end{theorem}

\begin{theorem} \cite[Th\'eor\`eme 15.4]{borel:groupes-arithmetiques} \label{borel-siegel-finite}
Let \( \Gamma \) be an arithmetic subgroup of \( \gG(\bQ) \).
Let \( \fS \subset \gG(\bR) \) be a Siegel set.

For any finite set \( C \subset \gG(\bQ) \) and any element \( \theta \in \gG(\bQ) \), the set
\[ \{ \gamma \in \Gamma : \gamma.C.\fS \cap \theta.C.\fS \neq \emptyset \} \]
is finite.
\end{theorem}

As remarked in section~\ref{ssec:siegel-sets-comparison}, \cref{borel-siegel-finite} requires the torus~\( \gS \) used in the definition of a Siegel set to satisfy condition~(ii) from section~\ref{ssec:siegel-sets-definition}, even though this condition is erroneously omitted from the statement in \cite{borel:groupes-arithmetiques}.

This paper's \cref{intro:main-theorem} implies \cite[Corollaire~15.3]{borel:groupes-arithmetiques} and therefore it implies \cref{borel-siegel-finite}, by the same argument as in the proof of \cite[Th\'eor\`eme 15.4]{borel:groupes-arithmetiques}.
Since our proof of \cref{intro:main-theorem} is independent of Borel's proof of \cite[Corollaire~15.3]{borel:groupes-arithmetiques}, this gives a new proof of \cref{borel-siegel-finite}.

\section{Proof of main height bound} \label{sec:gln:proof}

In this section we prove \cref{intro:main-theorem}.
Most of the section deals with the case of standard Siegel sets in \( \gGL_n \).
At the end we show how to deduce the general statement of \cref{intro:main-theorem} from this case, using \cref{intro:siegel-set-inclusion}.

Thus let \( \gG = \gGL_n \) and let \( \fS \) be a standard Siegel set in~\( \gG \).
As in the statement of \cref{intro:main-theorem}, we are given an element
\[ \gamma \in \fS.\fS^{-1} \cap \gG(\bQ), \]
with \( N = \abs{\det \gamma} \) and with \( D \) denoting the maximum of the denominators of entries of~\( \gamma \).
Since \( \gamma \in \fS.\fS^{-1} \), using the notation from section~\ref{ssec:gln:siegel-sets}, we can write
\begin{equation} \label{eqn:gln:master-gamma}
\gamma = \nu \beta \kappa \alpha^{-1} \mu^{-1}
\end{equation}
with \( \alpha, \beta \in A_t \), \( \mu, \nu \in \Omega_u \) and \( \kappa \in K \).
Rearranging this equation, we obtain
\begin{equation} \label{eqn:gln:master}
\gamma \mu \alpha = \nu \beta \kappa.
\end{equation}

Our aim is to bound the height of \( \gamma \) by a polynomial in \( N \) and~\( D \).
The proof has three stages.
First we compare entries of the diagonal matrices \( \alpha \) and~\( \beta \), showing that \( \alpha_j  \ll  D \beta_i \) for certain pairs of indices \( (i, j) \).
Secondly, we prove that
\begin{equation} \label{eqn:gln:b-ll-a}
\beta_j  \ll  N D^{n - 1} \alpha_i
\end{equation}
whenever \( i \) and~\( j \) lie in the same segment of a certain partition of \( \{ 1, \dotsc, n \} \).
Finally we expand out equation~\eqref{eqn:gln:master-gamma} and use inequality~\eqref{eqn:gln:b-ll-a}.

\subsection{Partitioning the indices}

An important device in the proof of \cref{intro:main-theorem} for standard Siegel sets is a partition of the set of indices \( \{ 1, \dotsc, n \} \) into subintervals which we call ``segments'' (depending on \( \gamma \)).
The \defterm{segments} are defined to be the subintervals of \( \{ 1, \dotsc, n \} \) such that:
\begin{enumerate}[(i)]
\item \( \gamma \) is block upper triangular with respect to the chosen partition;
\item \( \gamma \) is not block upper triangular with respect to any finer partition of \( \{ 1, \dotsc, n \} \) into subintervals.
\end{enumerate}

We define a \defterm{leading entry} to be a pair of indices \( (i, j) \in \{ 1, \dotsc, n \}^2 \) such that \( \gamma_{ij} \) is the leftmost non-zero entry in the \( i \)-th row of~\( \gamma \).

The following lemma describes segments in terms of leading entries.
This lemma also has a converse, which we will not need: if \( i > j \) and there exists a sequence satisfying condition~\eqref{cond:gln:segment-sequence}, then \( i \) and~\( j \) are in the same segment.

\begin{lemma} \label{gln:segment-witnessing-sequence}
If \( i > j \) and \( i \) and \( j \) are in the same segment, then there exists a sequence of leading entries \( (i_1, j_1), \dotsc, (i_s, j_s) \) such that
\begin{equation} \label{cond:gln:segment-sequence} \tag{*}
i \leq i_1, \quad j_p \leq i_{p+1} \text{ for every } p \in \{ 1, \dotsc, s-1 \}, \quad \text{ and } \; j_s \leq j.
\end{equation}
\end{lemma}

\begin{proof}
First, for each \( k \) such that \( j < k \leq i \),
we show that there exists a leading entry \( (i', j') \) such that \( j' < k \leq i' \).
Because segments give the finest partition according to which \( \gamma \) is block upper triangular, \( \gamma \) cannot be block upper triangular with respect to the partition
\[ \{ 1, \dotsc, k-1 \}, \{ k, \dotsc, n \}. \]
So there exists some \( i' \geq k \) such that the \( i' \)-th row of \( \gamma \) has a non-zero entry in the first \( k-1 \) columns.
Choosing \( j' \) to be the index of the leftmost non-zero entry in the \( i' \)-th row, we get the desired leading entry with \( j' < k \leq i' \).

Let \( s = i-j \).
For each \( p \) such that \( 1 \leq p \leq s \) we apply the above argument to \( k = i-p+1 \) and get a leading entry \( (i_p, j_p) \) such that \( j_p < i-p+1 \leq i_p \).
The resulting sequence \( (i_1, j_1), \dotsc, (i_s, j_s) \) satisfies condition~\eqref{cond:gln:segment-sequence}.
\end{proof}

We define \( \gQ \) to be the subgroup of \( \gGL_n \) consisting of block upper triangular matrices according to the segments defined above (thus \( \gQ \) depends on \( \gamma \)).
Observe that \( \gQ \) could equivalently be defined as the smallest standard parabolic subgroup of \( \gGL_n \) which contains \( \gamma \).

We define \( \gL \) to be the subgroup of \( \gGL_n \) consisting of block diagonal matrices according to the same partition into segments.
Thus \( \gL \) could equivalently be defined as the Levi subgroup of \( \gQ \) containing the torus of diagonal matrices.

\subsection{Example partitions for \texorpdfstring{\( \gGL_3 \)}{GL3}}

To illustrate the definition of segments and \cref{gln:segment-witnessing-sequence}, we show the various cases which occur for \( \gGL_3 \).
\Cref{tbl:components-gl3} shows classes of matrix in \( \gGL_3 \), depending on the region of zeros adjacent to the bottom left corner of the matrix, and gives the associated partitions of \( \{ 1, 2, 3 \} \) into segments.
Every matrix in \( \gGL_3 \) falls into exactly one of the classes in \cref{tbl:components-gl3}.

\begin{table}[b]
\caption{Partitions into segments for \( \gamma \in \gGL_3 \)}  \label{tbl:components-gl3}
\begin{tabular}{>{\renewcommand{\arraystretch}{1}}cc|>{\renewcommand{\arraystretch}{1}}cc}
   \( \gamma \)
 & Segments
 & \( \gamma \)
 & Segments

\\ \hline
&&&\\
   \( \begin{pmatrix}
	    *     & \cdot & \cdot
	 \\ 0     & *     & \cdot
	 \\ 0     & 0     & *
	 \end{pmatrix} \)
 & \( \{ 1 \} \), \( \{ 2 \} \), \( \{ 3 \} \)

 & \( \begin{pmatrix}
	    \cdot & \cdot & \cdot
	 \\ \cdot & \cdot & \cdot
	 \\ *     & \cdot & \cdot
	 \end{pmatrix} \)
 & \( \{ 1, 2, 3 \} \)
\\
&&&\\
   \( \begin{pmatrix}
	    *     & \cdot & \cdot
	 \\ 0     & \cdot & \cdot
	 \\ 0     & *     & \cdot
	 \end{pmatrix} \)
 & \( \{ 1 \} \), \( \{ 2, 3 \} \)

 & \( \begin{pmatrix}
	    \cdot & \cdot & \cdot
	 \\ *     & \cdot & \cdot
	 \\ 0     & *     & \cdot
	 \end{pmatrix} \)
 & \( \{ 1, 2, 3 \} \)
\\
&&&\\
  \( \begin{pmatrix}
	    \cdot & \cdot & \cdot
	 \\ *     & \cdot & \cdot
	 \\ 0     & 0     & *
	 \end{pmatrix} \)
 & \( \{ 1, 2 \} \), \( \{ 3 \} \)

 &&
\end{tabular}
\end{table}

In \cref{tbl:components-gl3}, \( * \) represents an entry which must be non-zero, while \( \cdot \) represents an entry which may be either zero or non-zero.
Every entry to the left of a~\( * \) is zero, so each~\( * \) is a leading entry.
For rows which do not contain a~\( * \), there is not enough information to determine the leading entry; these rows' leading entries rows are not important for \cref{gln:segment-witnessing-sequence}.

Comparing the two classes of matrices in the right-hand column of \cref{tbl:components-gl3}, we see that it is possible for matrices to have different patterns of zeros adjacent to the bottom left corner, yet still be associated with the same partition of \( \{ 1, 2, 3 \} \).
This is related to the fact that matrices in the lower class of this column do not form a subgroup of \( \gGL_3 \): the smallest standard parabolic subgroup containing such a matrix is the full group~\( \gGL_3 \), the same as for the upper class.

On the other hand, the difference between the two classes in the right-hand column of \cref{tbl:components-gl3} is important for finding sequences of leading entries as in \cref{gln:segment-witnessing-sequence}.
In the upper class of this column, the sequence consisting just of the leading entry \( (3, 1) \)
satisfies condition~\eqref{cond:gln:segment-sequence} for every pair~\( (i, j) \).
In the lower class, in order to construct a sequence satisfying condition~\eqref{cond:gln:segment-sequence} which goes from \( i=3 \) to \( j=1 \), we need both the leading entries \( (3,2) \) and \( (2,1) \).

\subsection{Ratios between diagonal matrices (leading entries)}

In the first stage of the proof, we compare \( \alpha_j \) with \( \beta_i \) when \( (i, j) \) is a leading entry.
This is based on comparing the lengths of the \( i \)-th rows on either side of equation~\eqref{eqn:gln:master}.

\begin{lemma} \label{gln:a-ll-b-leading-entry}
If \( (i, j) \) is a leading entry for \( \gamma \), then
\[
\alpha_j  \ll  D \beta_i.
\]
\end{lemma}

\begin{proof}
Recall equation~\eqref{eqn:gln:master}:
\[ \gamma \mu \alpha  =  \nu \beta \kappa. \]
Because \( \kappa \in \gO_n(\bR) \), multiplying by \( \kappa \) on the right does not change the length of a row vector.
Hence expanding out the lengths of the \( i \)-th rows on either side of \eqref{eqn:gln:master} gives
\begin{equation} \label{eqn:gln:row-length-master}
\sum_{p=1}^n \left( \sum_{q=1}^n \gamma_{iq} \mu_{qp} \right)^2 \alpha_p^2  =  \sum_{p=1}^n \nu_{ip}^2 \beta_p^2.
\end{equation}

Look first at the right hand side of equation~\eqref{eqn:gln:row-length-master}, comparing it to \( \beta_i^2 \).
Because \( \nu \) is upper triangular, non-zero terms on the right hand side of equation~\eqref{eqn:gln:row-length-master} must have \( p \geq i \)
and hence (by the definition of \( A_t \)) \( \beta_p \ll \beta_i \).
Since \( \nu \) is in the fixed compact set \( \Omega_u \), there is a uniform bound for the entries \( \nu_{ip} \).
Thus we get
\begin{equation} \label{ineqn:gln:nu-beta-length}
\sum_{p=1}^n \nu_{ip}^2 \beta_p^2  \ll  \beta_i^2.
\end{equation}

Now look at the left hand side of equation~\eqref{eqn:gln:row-length-master}, comparing it to \( \alpha_j^2 \).
We pull out the \( p = j \) term.
Because squares are nonnegative, we have
\begin{equation} \label{ineqn:gln:p-simplify-row-master-lhs}
\left( \sum_{q=1}^n \gamma_{iq} \mu_{qj} \right)^2 \alpha_j^2  \leq  \sum_{p=1}^n \left( \sum_{q=1}^n \gamma_{iq} \mu_{qp} \right)^2 \alpha_p^2.
\end{equation}

Because \( (i, j) \) is a leading entry, if \( \gamma_{iq} \neq 0 \) then \( q \geq j \).
Because \( \mu \) is upper triangular, if \( \mu_{qj} \neq 0 \) then \( q \leq j \).
Combining these facts, the only non-zero term on the left hand side of~\eqref{ineqn:gln:p-simplify-row-master-lhs} is the term with \( q = j \).
In other words,
\begin{equation} \label{eqn:gln:gamma-mu-row-length}
\gamma_{ij}^2 \mu_{jj}^2 \alpha_j^2  =  \left( \sum_{q=1}^n \gamma_{iq} \mu_{qj} \right)^2 \alpha_j^2.
\end{equation}

Because \( \mu \in \Omega_u \), we have \( \mu_{jj} = 1 \).
Because \( (i, j) \) is a leading entry, \( \gamma_{ij} \neq 0 \).
Because entries of \( \gamma \) are rational numbers with denominator at most \( D \), this implies that \( \abs{\gamma_{ij}} \geq D^{-1} \).
Combining these facts, we get
\begin{equation} \label{ineqn:gln:gm-squared}
D^{-2}  \leq  \gamma_{ij}^2 \mu_{jj}^2.
\end{equation}

Using successively the inequalities and equations \eqref{ineqn:gln:gm-squared}, \eqref{eqn:gln:gamma-mu-row-length}, \eqref{ineqn:gln:p-simplify-row-master-lhs}, \eqref{eqn:gln:row-length-master} and \eqref{ineqn:gln:nu-beta-length} gives
\[
D^{-2} \alpha_j^2  \ll  \beta_i^2.
\qedhere
\]
\end{proof}

\subsection{Ratios between diagonal matrices (in each segment)}

In the second stage of the proof of \cref{intro:main-theorem}, we prove a series of inequalities comparing entries of \( \alpha \) and \( \beta \).
This concludes with an inequality between \( \alpha_i \) and \( \beta_j \) valid whenever \( i \) and \( j \) are in the same segment.
(Note that the final inequality, \cref{gln:b-ll-a-same-segment}, is in the opposite direction to the starting point of \cref{gln:a-ll-b-leading-entry}.)

\begin{lemma} \label{gln:a-ll-b-same-index}
For all \( k \in \{ 1, \dotsc, n \} \),
\[
\alpha_k  \ll  D \beta_k.
\]
\end{lemma}

\begin{proof}
The key point is that there exists a leading entry \( (i, j) \) such that
\[ j \leq k \leq i. \]
To prove this, observe that since \( \gamma \) is invertible there must be some \( i \geq k \) such that the \( i \)-th row of \( \gamma \) contains a non-zero entry in or to the left of the \( k \)-th column.
Choosing \( j \) to be the index of the leftmost non-zero entry in the \( i \)-th row of~\( \gamma \) gives the required leading entry.

Taking such a leading entry \( (i, j) \), we can use \cref{gln:a-ll-b-leading-entry} (for the middle inequality) and the definition of~\( A_t \) (for the outer inequalities) to prove that
\[
\alpha_k  \ll  \alpha_j  \ll  D \beta_i  \ll  D \beta_k.
\qedhere
\]
\end{proof}

\begin{lemma} \label{gln:b-ll-a-multi-index}
For every set \( J \subset \{ 1, \dotsc, n \} \),
\[
\prod_{j \in J} \beta_j  \ll  N D^{n-\# J} \prod_{j \in J} \alpha_j.
\]
\end{lemma}

\begin{proof}
Because \( \alpha \) and~\( \beta \) are diagonal matrices with positive diagonal entries,
\begin{align}
         \prod_{j \in J} \beta_j \cdot \det \alpha
  &   =  \prod_{j \in J} \beta_j \cdot \prod_{k=1}^n \alpha_k
\notag
\\& \ll  D^{n-\# J} \prod_{j \in J} \alpha_j \cdot \prod_{k=1}^n \beta_k
      =  D^{n-\# J} \prod_{j \in J} \alpha_j \cdot \det \beta
\label{eqn:gln:prod-a-b-compare}
\end{align}
where the middle inequality uses \cref{gln:a-ll-b-same-index} for all indices~\( k \in \{ 1, \dotsc, n \} \setminus J \).

All of \( \mu \), \( \nu \) and \( \kappa \) have determinant \( \pm 1 \).
Hence equation~\eqref{eqn:gln:master} implies that
\[ \det \beta  =  N \det \alpha. \]
Combining this with inequality~\eqref{eqn:gln:prod-a-b-compare} proves the lemma.
\end{proof}

\begin{lemma} \label{gln:b-ll-a-same-segment}
If \( i \) and \( j \) are in the same segment, then
\[
\beta_j  \ll  N D^{n - 1} \alpha_i.
\]
\end{lemma}

\begin{proof}
If \( i \leq j \), then we apply \cref{gln:b-ll-a-multi-index} to the singleton \( \{ j \} \) to obtain
\[
\beta_j  \ll  N D^{n-1} \alpha_j.
\]
Combining this with \( \alpha_j  \ll  \alpha_i \) proves the lemma in the case \( i \leq j \).

Otherwise, \( i > j \) so we can use \cref{gln:segment-witnessing-sequence} to find a sequence of leading entries \( (i_1, j_1), \dotsc, (i_s, j_s) \) satisfying condition~\eqref{cond:gln:segment-sequence}.
We may assume that \( i_1, \dotsc, i_s \) are distinct -- otherwise we could simply delete the subsequence between two occurrences of the same \( i_p \).
Similarly, we may assume that none of \( i_1, \dotsc, i_s \) is equal to \( j \).

Therefore we can apply \cref{gln:b-ll-a-multi-index} to the set \( \{ i_1, \dotsc, i_s, j \} \) to get
\begin{equation} \label{ineqn:apply-multi-index}
\beta_j \prod_{p=1}^s \beta_{i_p}  \ll  N D^{n-(s+1)} \alpha_j \prod_{p=1}^s \alpha_{i_p}.
\end{equation}

For each \( p \in \{ 1, \dotsc, s-1 \} \), the fact that \( j_p \leq i_{p+1} \) and \cref{gln:a-ll-b-leading-entry} tell us that
\[
\alpha_{i_{p+1}}  \ll  \alpha_{j_p}  \ll  D \beta_{i_p}.
\]
Similarly because \( j_s \leq j \) we have
\[
\alpha_j  \ll  \alpha_{j_s}  \ll  D \beta_{i_s}.
\]
Multiplying these inequalities together and also multiplying by \( \beta_j \) gives the first inequality below, while \eqref{ineqn:apply-multi-index} gives the second:
\[
     \beta_j \alpha_j \prod_{p=2}^{s} \alpha_{i_p}  \ll  D^s \beta_j \prod_{p=1}^{s} \beta_{i_p}
\ll  N D^{n-1} \alpha_j \prod_{p=1}^s \alpha_{i_p}.
\]
Cancelling \( \alpha_j \prod_{p=2}^s \alpha_{i_p} \) shows that
\[ \beta_j  \ll  N D^{n-1} \alpha_{i_1}. \]
Since \( i \leq i_1 \), we have \( \alpha_{i_1}  \ll  \alpha_i \).
This completes the proof of the lemma.
\end{proof}

\subsection{Conclusion of proof for standard Siegel sets}

In the final stage of the proof, we expand out equation~\eqref{eqn:gln:master-gamma}.
When we do this, we get terms of the form \( \beta_p \kappa_{pq} \alpha_q^{-1} \).
In order to bound this using \cref{gln:b-ll-a-same-segment}, we need to know that \( \kappa_{pq} \) is zero if \( p \) and~\( q \) are not in the same segment.
In other words we have to begin by proving that \( \kappa \) is in the group~\( \gL(\bR) \) of block diagonal matrices.

\pagebreak

\begin{lemma} \label{gln:kappa-in-L} \label{gln:parabolic-intersect-compact}
\( \kappa \in \gL(\bR) \).
\end{lemma}

\begin{proof}
By construction, \( \gamma \), \( \mu \), \( \alpha \), \( \nu \), \( \beta \) are all in the group~\( \gQ(\bR) \) of block upper triangular matrices.
Hence equation~\eqref{eqn:gln:master-gamma} tells us that also \( \kappa \in \gQ(\bR) \).

If a matrix is both block upper triangular and orthogonal, then it is block diagonal according to the same blocks (because the inverse-transpose of a block upper triangular matrix is block lower triangular).
In other words,
\[ \gQ(\bR) \cap K  \subset  \gL(\bR). \]

This proves the lemma.
\end{proof}

\begin{lemma} \label{gln:gamma-bound}
For all \( i, j \in \{ 1, \dots, n \} \), we have
\[
\abs{\gamma_{ij}}  \ll  N D^{n - 1}.
\]
\end{lemma}

\begin{proof}
We expand out the matrix product in~\eqref{eqn:gln:master-gamma}, which we recall:
\[ \gamma = \nu \beta \kappa \alpha^{-1} \mu^{-1}. \]

Because \( \alpha \) and \( \beta \) are diagonal, the \( pq \)-th entry of \( \beta \kappa \alpha^{-1} \) is equal to
\[ \beta_p \kappa_{pq} \alpha_q^{-1}. \]
If \( p \) and \( q \) are not in the same segment, then \cref{gln:kappa-in-L} tells us that \( \kappa_{pq} = 0 \).
On the other hand if \( p \) and \( q \) are in the same segment, then we can apply \cref{gln:b-ll-a-same-segment} to bound \( \beta_p \alpha_q^{-1} \).
Furthermore, because \( \kappa \) is in the compact subgroup~\( K \), there is a uniform upper bound for entries of~\( \kappa \).
We conclude that
\begin{equation} \label{eqn:gln:bka}
\beta_p \kappa_{pq} \alpha_q^{-1}  \ll  N D^{n - 1}.
\end{equation}

Because \( \mu \) and \( \nu \) are in the fixed compact set~\( \Omega_u \) and because all elements of \( \Omega_u \) are invertible, there is a uniform upper bound for entries of \( \nu \) and of \( \mu^{-1} \).
Thus inequality~\eqref{eqn:gln:bka} together with equation~\eqref{eqn:gln:master-gamma} implies the lemma. 
\end{proof}

To complete the proof of \cref{intro:main-theorem} for standard Siegel sets in \( \gGL_n \), we just have to note that the definition of \( \rH(\gamma) \) implies that
\[ \rH(\gamma)  \leq  D \max(1, \abs{\gamma_{ij}}) \]
where the maximum is over all indices \( (i, j) \in \{ 1, \dotsc, n \}^2 \).
Hence \cref{gln:gamma-bound} implies that
\[
\rH(\gamma)  \leq  \max(D, \, \newC{height-bound-concl-multiplier} N D^n)
\]
where \( \refC{height-bound-concl-multiplier} \) denotes the implied constant from \cref{gln:gamma-bound}.

\subsection{Deducing general case from standard Siegel sets}
\label{ssec:gln:reduction}

To complete the proof of \cref{intro:main-theorem}, we deduce the general statement from the case of standard Siegel sets in~\( \gGL_n \).
This has two steps.
\Cref{gln:reduce-to-standard-siegel-set} allows us to generalise from standard Siegel sets to arbitrary Siegel sets in~\( \gGL_n \).
\Cref{intro:siegel-set-inclusion} (proved in section~\ref{sec:subgroups}) allows us to generalise from \( \gGL_n \) to arbitrary reductive groups~\( \gG \).

\begin{lemma} \label{gln:reduce-to-standard-siegel-set}
Let \( \fS \) be a Siegel set in \( \gGL_n(\bR) \).
Then there exist \( \gamma \in \gGL_n(\bQ) \) and \( \sigma \in \gGL_n(\bR) \) such that \( \gamma^{-1}.\fS.\gamma \sigma \) is contained in a standard Siegel set.
\end{lemma}

\begin{proof}
Let \( (\gP, \gS, K) \) be the Siegel triple associated with the Siegel set \( \fS \), and write
\( \fS = \Omega.A_t.K \) using the notation of section~\ref{ssec:siegel-sets-definition}.

Let \( (\gP_0, \gS_0, K_0) \) be the standard Siegel triple in \( \gGL_n \).
Write \( A_{0,t} \) and \( \Omega_{0,u} \) for the sets called \( A_t \) and \( \Omega_u \) in the definition of standard Siegel sets.

Since \( \gP \) and \( \gP_0 \) are minimal \( \bQ \)-parabolic subgroups of \( \gGL_n \), there exists \( \gamma \in \gGL_n(\bQ) \) such that \( \gP_0 = \gamma^{-1} \gP \gamma \).

Since \( K_0 \) and \( \gamma^{-1} K \gamma \) are maximal compact subgroups of \( \gGL_n(\bR) \), there exists \( \sigma \in \gGL_n(\bR) \) such that \( \gamma^{-1} K \gamma = \sigma K_0 \sigma^{-1} \).
Applying the Iwasawa decomposition
\[ \gGL_n(\bR) = \gU_0(\bR) . \gS_0(\bR)^+ . K_0, \]
we may assume that \( \sigma = \tau \beta \) where \( \beta \in \gS_0(\bR)^+ \) and \( \tau \in \gU_0(\bR) \).

Under this assumption, \( \sigma \in \gP_0(\bR) \).
Hence \( \sigma^{-1} \gamma^{-1}.\gP.\gamma \sigma = \gP_0 \).
By \cref{siegel-triple-unique-torus}, \( \sigma^{-1} \gamma^{-1}.\gS.\gamma \sigma = \gS_0 \).
Thus \( \sigma^{-1} \gamma^{-1}.A_t.\gamma \sigma = A_{0,t} \).

Now
\begin{align*}
      \gamma^{-1} \fS \gamma \sigma
  & = \gamma^{-1} \Omega \gamma . \sigma . \sigma^{-1} \gamma^{-1} A_t \gamma \sigma . \sigma^{-1} \gamma^{-1} K \gamma \sigma
\\& = \gamma^{-1} \Omega \gamma . \tau \beta . A_{0,t} . K_0
\end{align*}
Here \( \gamma^{-1} \Omega \gamma \tau \) is a compact subset of \( \gU_0(\bR) \) so it is contained in \( \Omega_{0,u} \) for a suitable \( u > 0 \).
Meanwhile \( \beta.A_{0,t} \) is contained in \( A_{0,s} \) for a suitable \( s > 0 \).
Thus \( \gamma^{-1} \fS \gamma \sigma \) is contained in the standard Siegel set \( \Omega_{0,u}.A_{0,s}.K_0 \), as required.
\end{proof}

\section{Siegel sets and subgroups} \label{sec:subgroups}

In this section we prove \cref{intro:siegel-set-inclusion}.
The proof gives additional information on the relationship between the Siegel triples for \( \gG \) and \( \gH \), as follows.

\begin{theorem} \label{siegel-set-inclusion}
Let \( \gG \) and \( \gH \) be reductive \( \bQ \)-algebraic groups, with \( \gH \subset \gG \).

Let \( \fS_\gH \) be a Siegel set in \( \gH(\bR) \) with respect to the Siegel triple \( (\gPH, \gSH, \KH) \).

Then there exist a Siegel set \( \fS_\gG \subset \gG(\bR) \) and a finite set \( C \subset \gG(\bQ) \) such that
\[ \fS_\gH \subset C.\fS_\gG. \]

Furthermore if \( (\gPG, \gSG, \KG) \) denotes the Siegel triple associated with \( \fS_\gG \), then
\( R_u(\gPH) \subset R_u(\gPG) \), \( \gSH = \gSG \cap \gH \) and \( \KH = \KG \cap \gH(\bR) \).
\end{theorem}

We denote sets used in the construction of the Siegel sets \( \fS_\gG \) and \( \fS_\gH \) by the notation from section~\ref{ssec:siegel-sets-definition} with the subscript \( \gG \) or \( \gH \) added as appropriate.
Thus we write
\[ \fS_\gH = \Omega_\gH . \AH{t} . \KH \] 
where \( \Omega_\gH \) is a compact subset of \( \gUH(\bR)\gMH(\bR)^+ \), \( \KH \) is a maximal compact subgroup of \( \gH(\bR) \) and
\[ \AH{t} = \{ \alpha \in \gSH(\bR)^+ : \chi(\alpha) \geq t \text{ for all } \chi \in \Delta_\gH \}. \]

\medskip

After the publication of this paper, Christian Schnell discovered an error in the proof of \cref{siegel-set-inclusion}, which has been corrected in this version of the paper.
Indeed, the original version of item~(2) below \cref{PH} was not strong enough for \cref{KZ-is-maximal-compact} to be valid.
I have therefore corrected item~(2) below \cref{PH}, and the proof of \cref{KZ-is-maximal-compact}, as suggested by Christian.
For additional explanation of this error, and examples showing that the conclusion of \cref{siegel-set-inclusion} may not be satisfied if we choose a subgroup~$\KG$ which does not satisfy the corrected item~(2), see the correction~\cite{orr-schnell}.

Experience since the publication of this paper has shown that \cref{siegel-set-inclusion} is often not sufficient for applications: one wants to choose $\KG$ in advance, rather than simply being assured that $\KG$ exists.
In fact, it is possible to choose $\KG$ in \cref{siegel-set-inclusion} to be any maximal compact subgroup of $\gG(\bR)$ satisfying the corrected item~(2) below \cref{PH}.
For a precise statement of this strengthened version of \cref{siegel-set-inclusion}, see \cite[Theorem~1]{orr-schnell}.

\subsection{Reduction to a split torus \texorpdfstring{\( \gSH \)}{SH}}

We begin by reducing the proof of \cref{siegel-set-inclusion} to the case in which the torus \( \gSH \) is \( \bQ \)-split.
Note that, even when \( \gSH \) is \( \bQ \)-split, it is not always possible to choose a \( \bQ \)-split torus for \( \gSG \).

According to the definition of a Siegel set, we can choose \( u \in \gPH(\bR) \) such that \( u \gSH u^{-1} \) is a maximal \( \bQ \)-split torus in \( \gPH \).
Using the Levi decomposition \( \gPH = Z_\gH(\gSH) \ltimes \gUH \), we may assume that \( u \in \gUH(\bR) \).

Now \( \Omega_\gH u^{-1} \) is a compact subset of \( \gUH(\bR).u\gMH(\bR)^+u^{-1} \) so
\[ \fS_\gH.u^{-1} = \Omega_\gH u^{-1} . u \AH{t} u^{-1} . u \KH u^{-1}. \]
is a Siegel set with respect to the Siegel triple \( (\gPH, u \gSH u^{-1}, u \KH u^{-1}) \).

We prove below that \cref{siegel-set-inclusion} holds when \( \gSH \) is \( \bQ \)-split.
Hence there exist a Siegel set \( \fS_\gG' \subset \gG(\bR) \) and a finite set \( C \subset \gG(\bQ) \) such that
\[ \fS_\gH.u^{-1} \subset C.\fS_\gG'. \]

Let \( (\gPG, \gSG', \KG') \) denote the Siegel triple associated with \( \fS_\gG' \).
According to \cref{siegel-set-inclusion}, \( \gUH \subset R_u(\gPG) \) and so \( u \in R_u(\gPG)(\bR) \).
Therefore
\[ \fS_\gG = \fS_\gG'.u \]
is a Siegel set for \( \gG(\bR) \) with respect to the Siegel triple \( (\gPG, u^{-1} \gSG' u, u^{-1} \KG' u) \).
We clearly have \( \fS_\gH \subset C.\fS_\gG \) and the Siegel triple associated with \( \fS_\gG \) satisfies the conditions of \cref{siegel-set-inclusion} relative to \( (\gPH, \gSH, \KH) \).

\subsection{Choosing the Siegel triple} \label{ssec:construct-siegel-triple}

We henceforth assume that \( \gSH \) is \( \bQ \)-split.
As the first step in proving \cref{siegel-set-inclusion} for this case, we choose a Siegel triple \( (\gPG, \gSG, \KG) \) for \( \gG \).

The main difficulty lies in choosing \( \gPG \).
The obvious idea is to choose a minimal parabolic \( \bQ \)-subgroup of \( \gG \) which contains \( \gPH \), but such a subgroup does not always exist (for example, if \( \gG \) is \( \bQ \)-split and \( \gH \) is \( \bQ \)-anisotropic).
Instead we construct a larger parabolic \( \bQ \)-subgroup \( \gQ \subset \gG \) which contains~\( \gPH \), and then define \( \gPG \) to be a minimal parabolic \( \bQ \)-subgroup of~\( \gQ \).

Let us write
\[ \gZ = Z_\gG(\gSH). \]

\begin{lemma} \label{construct-Q}
There exists a parabolic \( \bQ \)-subgroup \( \gQ \subset \gG \) such that
\begin{enumerate}[(i)]
\item \( \gZ \) is a Levi subgroup of \( \gQ \), and
\item \( \gUH \subset R_u(\gQ) \).
\end{enumerate}
\end{lemma}

\begin{proof}
Let \( \Phi_\gH^+ \) denote the set of roots \( \Phi(\gSH, \gPH) \).
By \cite[Proposition~3.1]{borel-tits:groupes-reductifs} there exists an order \( >_\gQ \) on \( X^*(\gSH) \) with respect to which all elements of \( \Phi_\gH^+ \) are positive.

Let
\[ \Phi_\gQ = \{ \chi \in \Phi(\gSH, \gG) : \chi >_\gQ 0 \} \]
and let \( \gQ \) denote the group \( \gG_{\Phi_\gQ} \) (using the notation of \cite[paragraph~3.8]{borel-tits:groupes-reductifs} with respect to the torus \( \gSH \)).
By \cite[Th\'eor\`eme~4.15]{borel-tits:groupes-reductifs}, \( \gQ \) is a parabolic \( \bQ \)-subgroup of~\( \gG \) and \( \gZ \) is a Levi subgroup of \( \gQ \).

Since all weights of \( \gSH \) on \( \gUH \) are contained in \( \Phi_\gH^+ \), which is a subset of \( \Phi_\gQ \), \cite[Proposition~3.12]{borel-tits:groupes-reductifs} tells us that \( \gUH \subset \gG_{\Phi_\gQ}^* \),
again using the notation of \cite[paragraph~3.8]{borel-tits:groupes-reductifs}.
By \cite[Th\'eor\`eme~3.13]{borel-tits:groupes-reductifs}, \( \gG_{\Phi_\gQ}^* = R_u(\gQ) \).
This completes the proof that \( \gUH \subset R_u(\gQ) \).
\end{proof}

We will make no use of the following lemma, but it sheds some light on the significance of the group \( \gQ \).

\begin{lemma} \label{PH}
\( \gPH = \gQ \cap \gH \).
\end{lemma}

\begin{proof}
We use the notation from the proof of \cref{construct-Q}.
By construction, we have that \( \Phi(\gSH, \gPH) = \Phi_\gH^+ \subset \Phi_\gQ \).
Hence by \cite[Proposition~3.12]{borel-tits:groupes-reductifs}, \( \gPH \subset \gG_{\Phi_\gQ} = \gQ \).

For the reverse inclusion, observe that \( \Phi(\gSH, \gQ \cap \gH) \subset \Phi_\gH^+ \).
Hence applying \cite[Proposition~3.12]{borel-tits:groupes-reductifs}, this time inside \( \gH \), we get
\[ \gQ \cap \gH \subset \gH_{\Phi_\gH^+} = \gPH.
\qedhere
\]
\end{proof}

\pagebreak

Choose the following subgroups of \( \gG \):
\begin{enumerate}
\item \( \gPG \), a minimal parabolic \( \bQ \)-subgroup of~\( \gQ \).
\item \( \KG \), a maximal compact subgroup of \( \gG(\bR) \) containing \( \KH \), such that the Cartan involution of~$\gG$ associated with $\KG$ stabilises $\gSH$.%
\footnote{Item~(2), the condition on~$\KG$, has been corrected from the published version of the paper.}
\end{enumerate}

\begin{lemma4A}\!%
\footnote{Lemma~4.A does not appear in the published version of the paper.}
There exists a maximal compact subgroup $\KG \subset \gG(\bR)$ satisfying the condition of item~(2) above.
\end{lemma4A}

\begin{proof}
Choose a faithful representation $\rho \colon \gG_\bR \to \gGL(V)$ for some real vector space~$V$.
By \cite[Theorem~7.3]{mostow:self-adjoint-groups}, there exists a positive definite symmetric form $\psi$ on~$V$ with respect to which the groups $\KH \subset \gH(\bR) \subset \gG(\bR) \subset \gGL(V)$ are simultaneously self-adjoint.
In other words, if $\Theta$ denotes the Cartan involution of $\gGL(V)$ associated with the form~$\psi$, then $\Theta$ restricts to Cartan involutions of $\KH$, $\gH$ and $\gG$.

Letting $\KG$ denote the stabiliser of $\psi$ in $\gG(\bR)$, we obtain $\KH \subset \KG$.

Since $\Theta$ restricts to the Cartan involution of $\gH$ associated with the maximal compact subgroup~$\KH$, and since $(\gPH, \gSH, \KH)$ is a Siegel triple for~$\gH$, $\Theta$ stabilises $\gSH$.
\end{proof}

Define the following notation for subgroups of \( \gG \) which are uniquely determined by \( \gPG \) and \( \KG \):
\begin{enumerate}
\item \( \gSG \) is the unique torus such that \( (\gPG, \gSG, \KG) \) is a Siegel triple for \( \gG \).
\item \( \gUG = R_u(\gPG) \).
\item \( \gPZ = \gPG \cap \gZ \) and \( \gUZ = R_u(\gPZ) \).
\item \( \KZ = \KG \cap \gZ(\bR) \).
\end{enumerate}

\begin{lemma} \label{KZ-is-maximal-compact}
\( \KZ \) is a maximal compact subgroup of \( \gZ(\bR) \).
\end{lemma}

\begin{proof}\!%
\footnote{The proof of \cref{KZ-is-maximal-compact} has been corrected from the published version of the paper.}
Let \( \Theta \) be the Cartan involution of \( \gG \) associated with the maximal compact subgroup \( \KG \).
By the condition on~$\KG$ in item~(2) above Lemma~4.A, \( \Theta \) stabilises \( \gSH \).
Hence \( \Theta \) also stabilises \( \gZ \).
Therefore the fixed points of \( \Theta \) in~\( \gZ(\bR) \), namely \( \KZ \), form a maximal compact subgroup of \( \gZ(\bR) \).
\end{proof}

\begin{lemma} \label{SH-in-SG}
\( \gSH \subset \gSG \).
\end{lemma}

\begin{proof}
Note that \( \gZ \) is a reductive group defined over~\( \bQ \), because \( \gSH \) is defined over~\( \bQ \).
Thus it makes sense to talk about Siegel triples in \( \gZ \).
By \cite[Proposition~4.4]{borel-tits:groupes-reductifs}, \( \gPZ \) is a minimal parabolic \( \bQ \)-subgroup of \( \gZ \).

By \cref{siegel-triple-unique-torus}, there exists a unique torus \( \gSZ \subset \gZ \) such that \( (\gPZ, \gSZ, \KZ) \) is a Siegel triple for \( \gZ \).
This means that:
\begin{enumerate}[(i)]
\item \( \gSZ \) is \( \gPZ(\bR) \)-conjugate to a maximal \( \bQ \)-split torus in \( \gPZ \).
Note that a maximal \( \bQ \)-split torus in \( \gPZ \) is also a maximal \( \bQ \)-split torus in \( \gPG \).

\item The Cartan involution of~\( \gZ \) associated with \( \KZ \) normalises \( \gSZ \).
This involution is the restriction of the Cartan involution of~\( \gG \) associated with~\( \KG \).
\end{enumerate}
Thus \( \gSZ \) satisfies the conditions of \cref{siegel-triple-unique-torus} with respect to \( (\gPG, \KG) \).
By the uniqueness in \cref{siegel-triple-unique-torus}, we conclude that \( \gSZ = \gSG \).

Because \( \gSZ \) is \( \gZ(\bR) \)-conjugate to a maximal \( \bQ \)-split torus in \( \gZ \), it contains every \( \bQ \)-split subtorus of the centre of~\( \gZ \).
In particular \( \gSH \subset \gSZ \).
\end{proof}

Let \( \gSG' \) be a maximal \( \bQ \)-split torus in \( \gPZ \).
Because \( (\gPZ, \gSZ, \KZ) \) is a Siegel triple, there exists \( u \in \gPZ(\bR) \) such that \( \gSG' = u \gSZ' u^{-1} \).
Because of the Levi decomposition \( \gPZ = Z_\gG(\gSG) \ltimes \gUZ \), we may assume that \( u \in \gUZ(\bR) \).

The following lemma is not needed in our proof of \cref{intro:siegel-set-inclusion}, but it contains extra information about \( \gSG \) which is included in the statement of \cref{siegel-set-inclusion}.

\begin{lemma}
\( \gSH = \gSG \cap \gH \).
\end{lemma}

\begin{proof}
Let \( q \) denote the quotient map \( \gPG \to \gPG/\gUG \).
Observe that \( \gUG \cap \gPH \) is a normal unipotent subgroup of \( \gPH \), so it is contained in \( \gUH \).
On the other hand,
\[ \gUH \subset R_u(\gQ) \cap \gPH \subset \gUG \cap \gPH. \]
Hence \( \gUG \cap \gPH = \gUH \), so \( q \) restricts to the quotient map \( \gPH \to \gPH/\gUH \).

According to the definition of a Siegel triple, \( q(\gSG) \) is a maximal \( \bQ \)-split torus in \( \gPG/\gUG \).
Furthermore, \( \gSG \cap \gH \subset \gQ \cap \gH = \gPH \).
Hence \( q(\gSG \cap \gH) \) is a \( \bQ \)-split torus in \( \gPH/\gUH \).

Since \( \gSH \subset \gSG \cap \gH \) and \( q(\gSH) \) is a maximal \( \bQ \)-split torus in \( \gPH/\gUH \),
we conclude that \( q(\gSH) = q(\gSG \cap \gH) \).
Because \( \gSG \cap \gUG = \{ 1 \} \), \( q_{|\gSG} \) is injective.
Thus \( \gSH = \gSG \cap \gH \).
\end{proof}

\subsection{Comparing \texorpdfstring{\( \AH{t} \)}{AHt} with \texorpdfstring{\( \AG{t'} \)}{AGt'}}

We now compare the sets \( \AH{t} \subset \gSH(\bR) \) and \( \AG{t'} \subset \gSG(\bR) \).
We would like to have \( \AH{t} \subset \AG{t'} \), but it is not always possible to choose \( t' \in \bR_{>0} \) such that this holds.
This is because there may be simple roots in \( \Phi(\gSG, \gG) \) whose restrictions to \( \gSH \) are not positive combinations of simple roots in \( \Phi(\gSH, \gH) \).
The values of such a root are bounded below by a positive constant on~\( \AG{t'} \) but can be arbitrarily close to zero on~\( \AH{t} \).

Instead we show that for a suitable value of \( t' \), every \( \alpha \in \AH{t'} \) can be conjugated into \( \AG{t'} \) by an element of the Weyl group \( N_\gG(\gSG) / Z_\gG(\gSG) \).
This element of the Weyl group must also satisfy certain other conditions which will be used later in the proof of \cref{siegel-set-inclusion}.

Write
\[ W = N_\gG(\gSG) / Z_\gG(\gSG), \quad W' = N_\gG(\gSG') / Z_\gG(\gSG'). \]
Since \( \gSG' = u \gSG  u^{-1} \), conjugation by \( u \) induces an isomorphism \( W \to W' \).

\pagebreak

\begin{proposition} \label{weyl-element-exists}
There exists \( t' > 0 \) (depending only on \( \gG \), \( \gH \), and \( t \)) such that for every \( \alpha \in \AH{t} \), there exists \( w \in W \) such that:
\begin{enumerate}[(i)]
\item \( \gUZ \subset w \gUG w^{-1} \),
\item \( \gUH \subset w \gUG w^{-1} \), and
\item \( \alpha \in w \AG{t'} w^{-1} \).
\end{enumerate}
\end{proposition}

Note that the statement of the proposition makes sense because \( w \gUG w^{-1} \) and \( w \AG{t'} w^{-1} \) do not depend on the choice of representative of \( w \) in \( N_\gG(\gSG) \).

\subsubsection*{Construction of \( \gQ_\alpha \)}

Suppose that we are given \( \alpha \in \AH{t} \).
In order to find \( w \in W \) as in \cref{weyl-element-exists},
we construct a parabolic subgroup \( \gP_{\gG,\alpha} = w \gPG w^{-1} \) by a refinement of the construction of \( \gPG \) from section~\ref{ssec:construct-siegel-triple}.
First we construct a larger parabolic subgroup \( \gQ_\alpha \) which satisfies conditions (i) and (ii) from \cref{construct-Q}, as well as the following additional condition:
\begin{enumerate}[(i)]
\setcounter{enumi}{2}
\item there exists \( t' > 0 \) (independent of \( \alpha \)) such that, for every \( \alpha \in \AH{t} \) and every \( \chi \in \Phi(\gSH, \gQ_\alpha) \), \( \chi(\alpha) \geq t' \).
\end{enumerate}
Similarly to the proof of \cref{construct-Q}, we construct \( \gQ_\alpha \) by choosing a suitable order~\( >_\alpha \) on \( X^*(\gSH) \).

Given \( \alpha \in \gSH(\bR)^+ \), choose a set \( \Psi_\alpha \subset \Phi(\gSH, \gG) \) which is maximal with respect to the following conditions:
\begin{enumerate}[(a)]
\item The set \( \Phi_\gH^+ \cup \Psi_\alpha \) is \( \bR_{>0} \)-independent.
(Recall that \( \Phi_\gH^+ = \Phi(\gSH, \gPH) \).)
\item For all \( \chi \in \Psi_\alpha \), \( \chi(\alpha) \geq 1 \).
\end{enumerate}
There always exists at least one set satisfying conditions (a) and (b), namely the empty set.
Since \( \Phi(\gSH, \gG) \) is finite, we deduce that there is a maximal set \( \Psi_\alpha \) satisfying the conditions.

By (a) there exists an order \( >_\alpha \) on \( X^*(\gSH) \) with respect to which all elements of \( \Phi_\gH^+ \cup \Psi_\alpha \) are positive.
Let
\[ \Phi_\alpha = \{ \chi \in \Phi(\gSH, \gG) : \chi >_\alpha 0 \} \]
and let \( \gQ_\alpha = \gG_{\Phi_\alpha} \) (in the notation of \cite[paragraph~3.8]{borel-tits:groupes-reductifs} with respect to \( \gSH \)).

The only condition on the order \( >_\gQ \) in the proof of \cref{construct-Q} was that all elements of \( \Phi_\gH^+ \) are positive with respect to \( >_\gQ \).
By definition, \( >_\alpha \) satisfies this condition.
Hence the proof of \cref{construct-Q} also applies to \( \gQ_\alpha \).
We conclude that \( \gQ_\alpha \) is a parabolic \( \bQ \)-subgroup of \( \gG \) satisfying conclusions (i) and (ii) of \cref{construct-Q}.

\begin{lemma} \label{Qalpha-roots-combinations}
Every root \( \chi \in \Phi_{\alpha} \) is a \( \bR_{>0} \)-combination of \( \Delta_\gH \cup \Psi_\alpha \).
\end{lemma}

\begin{proof}
If \( \chi \in \Psi_\alpha \), the result is trivial.
So we may assume that \( \chi \not\in \Psi_\alpha \).

Since \( \chi >_\alpha 0 \), \( \Psi_\alpha \cup \{ \chi \} \) satisfies (a).
Since \( \chi \not\in \Psi_\alpha \), the maximality of \( \Psi_\alpha \) tells us that \( \Psi_\alpha \cup \{ \chi \} \) does not satisfy (b).
Thus \( \chi(\alpha) < 1 \).

Hence \( \Psi_\alpha \cup \{ -\chi \} \) satisfies (b).
But \( -\chi <_\alpha 0 \), so \( -\chi \not\in \Psi_\alpha \).
Again by the maximality of \( \Psi_\alpha \), we conclude that \( \Psi_\alpha \cup \{ -\chi \} \) does not satisfy (a).
Thus there exist \( m_i, n_j, x \in \bR_{>0} \), \( \chi_i \in \Phi_\gH^+ \) and \( \psi_j \in \Psi_\alpha \) such that
\[ \sum_i m_i \chi_i + \sum_j n_j \psi_j + x(-\chi) = 0. \]
(The coefficient of \( -\chi \) in this equation must be non-zero because \( \Phi_\gH^+ \cup \Psi_\alpha \) is \( \bR_{>0} \)-independent.)

We can rearrange this equation to write \( \chi \) as a \( \bR_{>0} \)-combination of \( \Phi_\gH^+ \cup \Psi_\alpha \).
Since every element of \( \Phi_\gH^+ \) is a \( \bR_{>0} \)-combination of elements of~\( \Delta_\gH \), we deduce that \( \chi \) is a \( \bR_{>0} \)-combination of \( \Delta_\gH \cup \Psi_\alpha \).
\end{proof}

\begin{lemma} \label{Qalpha-root-bound}
There exists \( t' > 0 \) (depending on \( \gG \), \( \gH \) and \( t \) but not on \( \alpha \)) such that for every \( \alpha \in \AH{t} \) and every \( \chi \in \Phi_\alpha \), \( \chi(\alpha) \geq t' \).
\end{lemma}

\begin{proof}
Consider all pairs \( (\chi, \Xi) \) where \( \chi \in \Phi_\gG \) and \( \Xi \) is a subset of~\( \Phi_\gG \) such that \( \chi \) can be written as a \( \bR_{>0} \)-combination of elements of~\( \Xi \).
There are only finitely many such pairs, so we can find \( M \) (depending only on the root system \( \Phi_\gG \)) such that, for every such pair, there exist \( m_i \in \bR_{>0} \) and \( \xi_i \in \Xi \) satisfying
\[ \chi = \sum_i m_i \xi_i \text{ and } \sum_i m_i \leq M. \]

Suppose that \( \chi \in \Phi_\alpha \).
Using \cref{Qalpha-roots-combinations}, we can write \( \chi \) as a combination
\[ \chi = \sum_i m_i \chi_i + \sum_j n_j \psi_j \]
where \( \chi_i \in \Delta_\gH \), \( \psi_j \in \Psi_\alpha \), \( m_i, n_i \in \bR_{>0} \).
By the definition of \( M \), we may assume that \( \sum_i m_i + \sum_j n_j \leq M \).

By the definition of \( \AH{t} \), we have \( \chi_i(\alpha) \geq t \) for all \( i \).
By condition (b) on \( \Psi_\alpha \), we have \( \psi_j(\alpha) \geq 1 \) for all \( j \).
Therefore \( \chi(\alpha) \geq \min(1, t)^M \).
\end{proof}

\begin{proof}[Proof of \cref{weyl-element-exists}]
Because \( \gQ_\alpha \) satisfies conclusion~(i) of \cref{construct-Q}, \( \gZ \) is a Levi subgroup of \( \gQ_\alpha \).
Let \( \gP_{\gG,\alpha} = \gPZ \ltimes R_u(\gQ_\alpha) \).
By \cite[Proposition~4.4]{borel-tits:groupes-reductifs}, \( \gP_{\gG,\alpha} \) is a minimal \( \bQ \)-parabolic subgroup of \( \gG \).

By \cite[Corollaire~5.9]{borel-tits:groupes-reductifs},
the Weyl group \( W' \) acts transitively on the minimal parabolic \( \bQ \)-subgroups of~\( \gG \) containing the maximal \( \bQ \)-split torus \( \gSG' \).
Since \( \gSG' \subset \gPZ \subset \gP_{\gG,\alpha} \), we conclude that there exists \( w' \in W' \) (depending on \( \alpha \)) such that \( \gP_{\gG,\alpha} = w' \gPG w'^{-1} \).

Let \( w \) be the element of \( W \) which corresponds to \( w' \in W' \) via conjugation by \( u \).
Since \( u \in \gUZ(\bR) \subset \gPG(\bR) \cap \gP_{\gG,\alpha}(\bR) \), we have
\[ \gP_{\gG,\alpha} = w \gPG w^{-1}. \]
Since \( \gQ_\alpha \) satisfies conclusion (ii) of \cref{construct-Q}, we have
\[ \gUH \subset R_u(\gQ_\alpha) \subset R_u(\gP_{\gG,\alpha}) = w \gUG w^{-1}. \]
Furthermore \( \gPZ \subset \gP_{\gG,\alpha} \) and so \( \gUZ \subset R_u(\gP_{\gG,\alpha}) \).
This proves conclusions (i) and (ii) of \cref{weyl-element-exists}.

Since \( \gP_{\gG,\alpha} \subset \gQ_\alpha \), if \( \chi \in \Phi(\gSG, \gP_{\gG,\alpha}) \) then \( \chi_{|\gSH} \in  \Phi_\alpha \cup \{ 0 \} \).%
\footnote{Corrected from the published version.}
Hence by \cref{Qalpha-root-bound},
\[ \chi(\alpha) \geq t' \text{ for all } \alpha \in \AH{t} \text{ and }  \chi \in \Phi(\gSG,\gP_{\gG,\alpha}). \]
Noting that
\[ w \AG{t'} w^{-1} = \{ \beta \in \gSG(\bR)^+ : \chi(\beta) \geq t' \text{ for all simple roots of } \gP_{\gG,\alpha} \} \]
we conclude that \( \alpha \in w \AG{t'} w^{-1} \), proving conclusion (iii) of \cref{weyl-element-exists}.
\end{proof}

\subsection{Weyl group representatives} \label{ssec:weyl-group-representatives}

We need to choose two representatives for each element \( w \) in the Weyl group \( W = N_\gG(\gSG) / Z_\gG(\gSG) \).

Firstly we would like to choose representatives for \( W \) in \( \gG(\bQ) \).
However this is not usually possible because the torus \( \gSG \) is not defined over \( \bQ \).
Instead, recall that conjugation by \( u \) induces an isomorphism \( W \to W' \).
Given \( w \in W \), let \( w' \) denote the corresponding element of \( W' \).
By \cite[Th\'eor\`eme~5.3]{borel-tits:groupes-reductifs},
we can choose \( w_\bQ' \in \gG(\bQ) \) which represents \( w' \).
We then get a representative for \( w \) by setting
\[ w_\bQ =  u^{-1} \, w_\bQ' \, u. \]

Secondly we choose representatives for \( W \) in \( \KG \).

\begin{lemma} \label{w-k}
Let \( \gG \) be a reductive \( \bQ \)-algebraic group.
Let \( (\gPG, \gSG, \KG) \) be a Siegel triple in \( \gG \).

Every \( w \in N_\gG(\gSG) / Z_\gG(\gSG) \) has a representative \( w_K \in \KG \).
\end{lemma}

\begin{proof}
Let $\gTG$ be a maximal $\bR$-split torus in~$\gG$ which contains $\gSG$ and is stabilised by the Cartan involution.%
\footnote{The first sentence of the proof of \cref{w-k} has been corrected from the published version.}

Let \( \gN = N_\gG(\gSG) \cap N_\gG(\gTG) \).
Because \( \gSG \) is conjugate to a maximal \( \bQ \)-split torus of \( \gG \), \cite[Corollaire~5.5]{borel-tits:groupes-reductifs} implies that
\[ N_\gG(\gSG) = \gN.Z_\gG(\gSG). \]
Therefore we can choose \( \sigma \in \gN(\bC) \) such that \( w = \sigma.Z_\gG(\gSG) \).

According to the final displayed equation from \cite[section~14]{borel-tits:groupes-reductifs},
every element of \( N_\gG(\gTG) / Z_\gG(\gTG) \) has a representative in \( \KG \).
In particular, there exists \( w_K \in N_\gG(\gTG)(\bR) \cap \KG \) which represents \( \sigma.Z_\gG(\gTG) \).
Then
\[ w_K \sigma^{-1} \in Z_\gG(\gTG)(\bC) \subset Z_\gG(\gSG)(\bC). \]
It follows that \( w_K \) normalises \( \gSG \) and represents \( w \in N_\gG(\gSG) / Z_\gG(\gSG) \).
\end{proof}

Since the Cartan involution of \( \gG \) associated with \( \KG \) stabilises \( \gSG \), it also stabilises \( Z_\gG(\gSG) \).
Hence \( \KG \cap Z_\gG(\gSG)(\bR) \) is a maximal compact subgroup of \( Z_\gG(\gSG)(\bR) \).
By \cite[Chapter~XV, Theorem~3.1]{hochschild:lie-groups}, \( \KG \cap Z_\gG(\gSG)(\bR) \) meets every connected component of \( Z_\gG(\gSG)(\bR) \).
When choosing \( w_K \) as in \cref{w-k}, we may therefore assume that \( w_K \in w_\bQ.Z_\gG(\gSG)(\bR)^+ \).

We will need the following lemma about \( w_\bQ \) and \( w_\bQ' \).
This lemma does not hold for every element of \( W \), so we restrict our attention to elements which satisfy conditions (i) and (ii) of \cref{weyl-element-exists}, that is, elements of the set
\[ W^\dag = \{ w \in W : \gUZ \subset w \gUG w^{-1} \text{ and } \gUH \subset w \gUG w^{-1} \}. \]

\begin{lemma} \label{prod-of-wqs}
If \( w \in W^\dag \), then \( w_\bQ'^{-1} w_\bQ \in \gUG(\bR) \).
\end{lemma}

\begin{proof}
By definition,
\[ w_\bQ'^{-1} w_\bQ = u w_\bQ^{-1} u^{-1} w_\bQ. \]
Because \( w \in W^\dag \) and \( u \in \gUZ(\bR) \), we have
\[ w_\bQ^{-1} u^{-1} w_\bQ \in \gUG(\bR). \]
Multiplying this by \( u \in \gUG(\bR) \) proves the lemma.
\end{proof}

\subsection{Construction of the compact set \texorpdfstring{\( \Omega_\gG \)}{OmegaG}} \label{ssec:omegaG}

By the Langlands decomposition in \( \gPH \), the multiplication map
\[ \gUH(\bR) \times \gMH(\bR)^+ \to \gUH(\bR).\gMH(\bR)^+ \]
is a homeomorphism.
Hence there exist compact sets \( \Omega_{\gUH} \subset \gUH(\bR) \) and \( \Omega_{\gMH} \subset \gMH(\bR)^+ \) such that
\begin{equation} \label{eqn:decomposition-omega-h}
\Omega_\gH \subset \Omega_{\gUH} . \Omega_{\gMH}.
\end{equation}

Since \( \gMH \) need not be contained in \( \gMG \), we need to further decompose \( \Omega_{\gMH} \).
Let \( \gBZ \) be a minimal \( \bR \)-parabolic subgroup of \( \gZ = Z_\gG(\gSH) \) contained in \( \gPZ \).
By the Iwasawa decomposition in \( \gZ \), the multiplication map
\( \gBZ(\bR)^+ \times \KZ \to \gZ(\bR) \)
is a homeomorphism so there exists a compact set \( \Omega_{\gBZ} \subset \gBZ(\bR)^+ \) such that
\begin{equation} \label{eqn:decomposition-omega-mh}
\Omega_{\gMH} \subset \Omega_{\gBZ}.\KZ.
\end{equation}

For each \( w \in W^\dag \), choose \( w_K \), \( w_\bQ \) and \( w_\bQ' \) as in section~\ref{ssec:weyl-group-representatives}.
We have \( w_K w_\bQ^{-1} \in Z_\gG(\gSG)(\bR)^+ \subset \gPZ(\bR)^+ \) and \( \gBZ(\bR)^+ \subset \gPZ(\bR)^+ \), so
\( \Omega_{\gBZ}.w_K w_\bQ^{-1} \) is a compact subset of \( \gPZ(\bR)^+ \).
Noting that \( Z_\gG(\gSG) \) is a Levi subgroup of \( \gPZ \), the Langlands decomposition in \( \gPZ \) \cite[equation~(I.1.8)]{borel-ji:symmetric-spaces} tells us that the multiplication map
\[ \gUZ(\bR) \times \gMG(\bR)^+ \times \gSG(\bR)^+ \to \gPZ(\bR)^+ \]
is a homeomorphism.
Therefore there exist compact sets \( \Omega_{\gUZ}^{[w]} \subset \gUZ(\bR) \), \( \Omega_{\gMG}^{[w]} \subset \gMG(\bR)^+\) and \( \Omega_{\gSG}^{[w]} \subset \gSG(\bR)^+ \) such that
\begin{equation} \label{eqn:decomposition-omega-bz}
\Omega_{\gBZ}.w_K w_\bQ^{-1}  \subset  \Omega_{\gUZ}^{[w]}.\Omega_{\gMG}^{[w]}.\Omega_{\gSG}^{[w]}.
\end{equation}

Let
\[ \Omega_\gG = \bigcup_{w \in W^\dag} w_\bQ'^{-1} . \Omega_{\gUH} . \Omega_{\gUZ}^{[w]} . \Omega_{\gMG}^{[w]} . w_\bQ. \]
Since \( W^\dag \) is finite, \( \Omega_\gG \) is compact.

\pagebreak

\begin{lemma}
\( \Omega_\gG \subset \gUG(\bR) \gMG(\bR)^+ \).
\end{lemma}

\begin{proof}
For each \( w \in W^\dag \), by \cref{prod-of-wqs}, \( w_\bQ'^{-1} w_\bQ \in \gUG(\bR) \).
Using the definition of \( W^\dag \), we have
\[ w_\bQ^{-1}  \Omega_{\gUH}  w_\bQ  \subset  \gUG(\bR)  \quad \text{ and } \quad  w_\bQ^{-1} \Omega_{\gUZ}^{[w]} w_\bQ  \subset  \gUG(\bR). \]
Multiplying these together, we conclude that
\begin{equation} \label{eqn:compact-in-UG}
w_\bQ'^{-1} . \Omega_{\gUH} . \Omega_{\gUZ}^{[w]} . w_\bQ  \subset  \gUG(\bR).
\end{equation}

Since \( \gSG \) is \( \gG(\bR) \)-conjugate to a maximal \( \bQ \)-split torus in \( \gG \), we can use \cite[Corollaire~5.4]{borel-tits:groupes-reductifs} to show that \( \gMG \) is normal in \( N_\gG(\gSG) \).
It follows that \( w_\bQ \) normalises \( \gMG(\bR)^+ \) and so
\begin{equation} \label{eqn:compact-in-MG}
w_\bQ^{-1} \Omega_{\gMG}^{[w]} w_\bQ  \subset  \gMG(\bR)^+.
\end{equation}

Combining \eqref{eqn:compact-in-UG} and \eqref{eqn:compact-in-MG} proves the lemma.
\end{proof}

\begin{lemma} \label{conjugate-of-omega}
For each \( w \in W^\dag \),
\( w_\bQ'^{-1} \Omega_\gH  \subset  \Omega_\gG . w_K^{-1} . \Omega_{\gSG}^{[w]} .\KZ \).
\end{lemma}

\begin{proof}
Noting that \( w_\bQ w_K^{-1} \) commutes with \( \gSG \), we can rearrange \eqref{eqn:decomposition-omega-bz} to obtain
\[ \Omega_{\gBZ}  \subset  \Omega_{\gUZ}^{[w]}.\Omega_{\gMG}^{[w]}.w_\bQ w_K^{-1}.\Omega_{\gSG}^{[w]}. \]
Combining this with \eqref{eqn:decomposition-omega-h} and \eqref{eqn:decomposition-omega-mh}, we get
\begin{align*}
    \Omega_\gH
  & \subset  \Omega_{\gUH}.\Omega_{\gMH}  \subset  \Omega_{\gUH}.\Omega_{\gBZ}.\KZ
\\& \subset  \Omega_{\gUH} . \Omega_{\gUZ}^{[w]}.\Omega_{\gMG}^{[w]}.w_\bQ w_K^{-1}.\Omega_{\gSG}^{[w]} . \KZ.
\end{align*}
We can now read off the lemma using the definition of \( \Omega_\gG \).
\end{proof}

\subsection{The Siegel set for \texorpdfstring{\( \gG \)}{G}}

For each \( w \in W^\dag \), \( w_K^{-1} \Omega_{\gSG}^{[w]} w_K \) is a compact subset of \( \gSG(\bR)^+ \).
Hence there exists \( s > 0 \) such that \( \chi(\beta) \geq s \) for all \( \chi \in \Delta_\gG \) and all \( \beta \in w_K^{-1} \Omega_{\gSG}^{[w]} w_K \)
(since \( W^\dag \) is finite, we can choose a single value of \( s \) which works for all \( w \in W^\dag \)).

Let \( \fS_\gG \) be the Siegel set
\[ \fS_\gG = \Omega_\gG . \AG{t's} . \KG \subset \gG(\bR), \]
using \( t' \) from \cref{weyl-element-exists} and \( \Omega_\gG \) from section~\ref{ssec:omegaG}.
Let \( C \) be the finite set
\[ C = \{ w_\bQ' : w \in W^\dag \} \subset \gG(\bQ). \]

\begin{proposition}
\( \fS_\gH \subset C.\fS_\gG \).
\end{proposition}

\begin{proof}
Given \( \sigma \in \fS_\gH \), we can write
\[ \sigma = \mu \alpha \kappa \]
with \( \mu \in \Omega_\gH \), \( \alpha \in \AH{t} \) and \( \kappa \in \KH \).

By \cref{weyl-element-exists}, we can choose \( w \in W^\dag \) such that \( \alpha \in w \AG{t'} w^{-1} \).
By \cref{conjugate-of-omega}, we can write
\[ w_\bQ'^{-1} \mu = \nu w_K^{-1} \beta \lambda \]
where \( \nu \in \Omega_\gG \), \( \beta \in \Omega_{\gSG}^{[w]} \) and \( \lambda \in \KZ \).
Therefore
\[ w_\bQ'^{-1} \sigma = \nu w_K^{-1} \beta \lambda \, \alpha \kappa. \]
Since \( \lambda \in \KZ \subset \gZ(\bR) \), \( \lambda \) commutes with \( \alpha \in \gSH(\bR) \) so we can rewrite this as
\[ w_\bQ'^{-1} \sigma = \nu . w_K^{-1} \, \beta \alpha w_K . w_K^{-1} \lambda \kappa. \]

By definition, \( \nu \in \Omega_\gG \).
By the definition of~\( s \), we have \( w_K^{-1} \beta w_K \in \AG{s} \) while \( w_K^{-1} \alpha w_K \in \AG{t'} \) by \cref{weyl-element-exists}.
Hence
\[ w_K^{-1} \beta \alpha w_K \in \AG{t's}. \]
Finally, \( w_K^{-1} \), \( \lambda \) and \( \kappa \) are all in the group \( \KG \), so their product is also in \( \KG \).

Thus we have shown that \( w_\bQ'^{-1} \sigma \in \fS_\gG \), and so \( \sigma \in C.\fS_\gG \).
\end{proof}

\bibliographystyle{amsalpha}
\bibliography{reduction2}

\end{document}